\documentclass{amsart}
\usepackage[utf8]{inputenc}
\usepackage{amsmath}
\usepackage{amssymb}
\usepackage{amsthm}
\usepackage{mathtools}
\usepackage{color} 
\definecolor{darkblue}{rgb}{0,0,0.6}
\usepackage[breaklinks, pdftex, ocgcolorlinks,colorlinks=true, citecolor=darkblue, filecolor=darkblue, linkcolor=darkblue, urlcolor=darkblue]{hyperref}
\newcounter{commentcounter}

\title[The FJC for hyperbolic and CAT(0)-gropus revisited]{The Farrell--Jones conjecture for hyperbolic and CAT(0)-groups revisited}
\author{Daniel Kasprowski and Henrik R\"uping}
\address{Max-Planck-Institut f\"ur Mathematik, Vivatsgasse 7, 53111 Bonn, Germany}
\email{kasprowski@mpim-bonn.mpg.de}

\newcommand{\IR}{\mathbb{R}}
\newcommand{\bbR}{\mathbb{R}}

\newcommand{\bbH}{\mathbb{H}}

\newcommand{\bbZ}{\mathbb{Z}}
\newcommand{\bbN}{\mathbb{N}}

\newcommand{\ignore}[1]{}
\newcommand{\FJCw}{{FJCw}}
\DeclareMathOperator{\id}{id}

\DeclareMathOperator{\Map}{Map}
\DeclareMathOperator{\SL}{SL}

\DeclareMathOperator{\SO}{SO}

\DeclareMathOperator{\Lip}{Lip}
\DeclareMathOperator{\res}{res}

\newcommand{\VCyc}{{\mathcal{V}cyc}}

\newcommand{\calu}{\mathcal{U}}
\newcommand{\cala}{\mathcal{A}}

\renewcommand{\epsilon}{\varepsilon}

\numberwithin{equation}{section}
\newtheorem{thm}{Theorem}[section]
\newtheorem{prop}{Proposition}[section]

\newtheorem{corollary}{Corollary}[section]
\newtheorem{lemma}{Lemma}[section]
\theoremstyle{definition}
\newtheorem{assumption}{Assumption}[section]
\newtheorem{example}{Example}[section]
\newtheorem{defi}{Definition}[section]
\newtheorem{definition}{Definition}[section]
\newtheorem{rem}{Remark}[section]
\newtheorem{remark}{Remark}[section]

\makeatletter
\let\c@lemma=\c@thm
\let\c@prop=\c@thm
\let\c@cor=\c@thm
\let\c@corollary=\c@thm
\let\c@assumption=\c@thm
\let\c@example=\c@thm
\let\c@defi=\c@thm
\let\c@definition=\c@thm
\let\c@remark=\c@thm
\let\c@rem=\c@thm
\let\c@nota=\c@thm
\makeatother

 \newtheoremstyle{TheoremNum}
        {}{}              
        {\itshape}                      
        {}                              
        {\bfseries}                     
        {.}                             
        { }                             
        {\thmname{#1}\thmnote{ \bfseries #3}}
\theoremstyle{TheoremNum}

\address{The University of British Columbia, Department of Mathematics, 1984 Mathematics Road, Vancouver, B.C., Canada V6T 1Z2}
\email{rueping@math.ubc.ca}
\date{\today}
\begin{document}
\begin{abstract}
We generalize the proof of the Farrell--Jones conjecture for CAT(0)-groups to a larger class of groups, for example also containing all groups that act properly and cocompactly on a finite product of hyperbolic graphs. In particular, this gives a unified proof of the Farrell--Jones conjecture for CAT(0)- and hyperbolic groups.
\end{abstract}
\maketitle
\section{Introduction}
The Farrell--Jones conjecture for a group $G$  says that the $K$-theoretic assembly map
\[\mathcal{H}^G_*(E_\VCyc G;\mathbf{K}_\cala)\rightarrow \mathcal{H}^G_*(pt;\mathbf{K}_\cala)=K^{alg}_*(\cala[G])\]
and the $L$-theoretic assembly map
\[\mathcal{H}^G_*(E_\VCyc G;\mathbf{L}_\cala)\rightarrow \mathcal{H}^G_*(pt ;\mathbf{L}_\cala)=L^{\langle-\infty\rangle}_*(\cala[G])\]
are isomorphisms for any additive $G$-category $\cala$ (with involution), see Bartels and Reich \cite[Conjectures~3.2~and~5.1]{coefficients}. The Farrell--Jones conjecture implies several other conjectures. More background information about the Farrell--Jones conjecture can be found in L\"uck and Reich \cite{baum}.

As in \cite[Definition~2.15]{wegner2013farrell} we say that a group $G$ satisfies the Farrell--Jones conjecture with finite wreath products if for any finite group $F$ the wreath product $G\wr F$ satisfies the $K$- and $L$-theoretic Farrell--Jones conjecture. We will use the abbreviation \FJCw{} for "Farrell--Jones conjecture with finite wreath products".

We will weaken the assumptions in the proof of the Farrell-Jones conjecture for CAT(0)-groups (see Bartels and L\"uck \cite{bartels2012borel,flow} and Wegner \cite{wegnercat}). In \autoref{sec:bicomb} we begin by defining some properties of bicombings. The assumptions of our main theorem, \autoref{thm:main}, are stated in \autoref{assumption}. We also give some examples of groups satisfying those. As one application we obtain the following result.
\begin{thm}
\label{thm:prodgraphs}
All groups acting properly and cocompactly on a finite products of hyperbolic graphs satisfy the Farrell--Jones conjecture with finite wreath products.
\end{thm}
The Farrell--Jones conjecture is known for hyperbolic groups by work of Bartels, L\"uck and Reich \mbox{\cite{bartels2008equivariant,bartels2008k}}. But \autoref{thm:main} now gives a unified proof of the Farrell--Jones conjecture working for both hyperbolic and CAT(0)-groups. In \autoref{sec:noncat} we give an example of a group satisfying \autoref{assumption} which is not a CAT(0)-group.

In \autoref{sec:flow} and \autoref{sec:trans} we generalize several results from Bartels and L\"uck \cite{flow} about flow spaces and contracting transfers to our setting. Using this we can prove the main theorem in \autoref{sec:main}.

\textbf{Acknowledgements:} We would like to thank Arthur Bartels for helpful discussions and Martin Bridson for telling us about the example in \autoref{sec:noncat}. We also would like to thank Daniel L\"utgehetmann, Tessa Turini, Mark Ullmann and the referee for useful comments and suggestions. The first author was supported by the Max Planck Society. 
\section{Bicombings}
\label{sec:bicomb}
\begin{defi}\label{def:alot}
A \emph{constant-speed bicombing} on a metric space $X$ is a continuous function
\[\gamma\colon X\times X\times [0,1]\to X\]
such that for every $x,y\in X$ the function $\gamma_{x,y}\colon [0,1]\to X$ with $\gamma_{x,y}(t)=\gamma(x,y,t)$ is a rectifiable path of constant speed $l(\gamma_{x,y})$ from $x$ to $y$.

The bicombing $\gamma$ is called 
\begin{enumerate}
\item \emph{geodesic} if those paths are actually geodesics, i.e., $l(\gamma_{x,y})=d(x,y)$; 
\item \emph{convex} if for every $x,x',y,y'\in X$ the function
\[[0,1]\to\bbR,\qquad t\mapsto d(\gamma_{x,y}(t),\gamma_{x',y'}(t))\]
is convex;
\item \emph{$A$-convex} for a continuous function $A\colon [0,1]\times [0,\infty)\times [0,\infty)\rightarrow [0,\infty)$ with $A(1,s,0)=A(0,0,s)=0$ for all $s\geq0$ if for every $x,x',y,y'\in X$ and $t\in[0,1]$ we have
\[d(\gamma_{x,y}(t),\gamma_{x',y'}(t))\leq A(t,d(x,x'),d(y,y'));\]
\item \emph{consistent} if the chosen paths behave well under restriction, i.e. if for all $x,y\in X,s<s'\in[0,1],t\in[s,s']$ we have
\[\gamma_{x,y}(t)=\gamma_{\gamma_{x,y}(s),\gamma_{x,y}(s')}\left(\tfrac{t-s}{s'-s}\right).\]
\end{enumerate}
Convex bicombings are $A$-convex for the function $A(t,s,s')=(1-t)s+ts'$.
Let $G$ be a group acting isometrically on $X$, then the bicombing $\gamma$ is \emph{equivariant} if for every $g\in G, x,y\in X, t\in[0,1]$ we have
\[g\gamma_{x,y}(t)=\gamma_{gx,gy}(t).\]
\end{defi}
\begin{remark} In the literature it is not always assumed that a bicombing is continuous in $X\times X$. 
\end{remark}
\begin{remark}
The definition of $A$-convex is made in such a way that two chosen paths with the same endpoint eventually are close together in the following sense.
If a bicombing $\gamma$ is $A$-convex , then for every $s,\epsilon>0$ there exists $\delta>0$ such that for all $t\in[1-\delta,1],x,x',y\in X$ with $d(x,x')\leq s$ we have
\[d(\gamma_{x,y}(t),\gamma_{x',y}(t))\leq A(t,d(x,x'),0)<\epsilon.\]
\end{remark}
We define the following generalization of CAT(0)-groups. The name is motivated by the definition of a Busemann space, which is a geodesic space such that all pairs of geodesics are convex, see \cite[Section 1]{bowditch}. Note that we will only assume convexity for a bicombing.
\begin{defi}
A group $G$ is called \emph{Busemann group} if there exists a finite-dimensional, proper metric space $X$ with a cocompact, proper and isometric $G$-action and a consistent, convex, equivariant, geodesic bicombing $\gamma$ on $X$.
\end{defi}
\begin{example}
\label{ex:bus}~
\begin{enumerate}
\item Every CAT(0)-group is a Busemann group.
\item By a result of Descombes and Lang \cite[Theorem 1.3]{lang2} every hyperbolic group is a Busemann group. Namely, the action on its injective hall satisfies the assumption. See \autoref{ex:prod} for a quick review of their proof. 
\end{enumerate}
\end{example}
We will proof our main theorem under the following even more generally but also more technical assumption. In the next section we will give an example of a group satisfying it, which is not a CAT(0)-group.
\begin{assumption}
\label{assumption}
Let $G$ be a group. Assume there exists a finite-dimensional, proper metric space $X$ with a cocompact, proper and isometric $G$-action. Furthermore, assume there exists a consistent, $A$-convex, constant-speed, equivariant bicombing $\gamma$ on $X$ with $\gamma_{x,x}(t)=x$ for all $x\in X, t\in[0,1]$ and a continuous function $f\colon[0,\infty)\to[0,\infty)$ with $f(0)=0$ and $|l(\gamma_{x,y})-l(\gamma_{x',y'})|\leq f(d(x,x')+d(y,y'))$ for all $x,x',y,y'\in X$.
\end{assumption}
For the proof of the main theorem we will need certain monotonicity assumptions on the functions $A$ and $f$ appearing in \autoref{assumption}. Those can always be satisfied by the following remark.
\begin{remark} 
\label{rem:monoton}
Given a function $A$ as in the definition of $A$-convexity.

Let $B\colon [0,1]\times [0,\infty)\times[0,\infty) \rightarrow [0,\infty)$ be given by
\[(t,s,s')\mapsto\max\{A(t,r,r')\mid 0\leq r\leq s, 0\leq r'\leq s'\}.\]
We can always replace $A$ by 
\[A'\colon [0,1]\times [0,\infty)\times[0,\infty) \rightarrow [0,\infty),\]\[(t,s,s')\mapsto
\begin{cases}
\max\{B(t',s,s')\mid 0\leq t'\leq t\}&0\leq t\leq 1/3\\
\max\{B(t',s,s')\mid t\leq t'\leq 1\}&2/3\leq t\leq 1\\
\max\{B(t,s,s'),(3t-1)A'(2/3,s,s')&\\ \qquad+(2-3t)A'(1/3,s,s')\}&1/3\leq t\leq 2/3
\end{cases}.\]
So without loss of generality we can assume that $A$ is monotonically increasing in the second and third coordinate, in the first coordinate monotonically increasing on $[0,1/3]$ and decreasing on $[2/3,1]$.

By the same argument we can assume that $f$ is monotonically increasing.
\end{remark}
\begin{lemma}
Let $(X,d),(X',d')$ be metric spaces with constant-speed bicombings $\gamma$ and $\gamma'$ respectively. Then there exists a constant-speed bicombing $\overline{\gamma}$ on the product $X\times X'$ with the $l^2$-metric which is consistent if both $\gamma$ and $\gamma'$ are. If $\gamma$ and $\gamma'$ are $A$- and $A'$-convex respectively, where $A$ and $A'$ satisfy the monotonicity assumptions from \autoref{rem:monoton}, then $\overline \gamma$ is $\overline A\coloneqq(A^2+(A')^2)^{1/2}$-convex.
\end{lemma}
\begin{proof}
Define 
\[\overline{\gamma}\colon (X\times X')^2\times[0,1] \rightarrow X\times X',~((x,x'),(y,y'),t)\mapsto (\gamma(x,y,t),\gamma'(x,',y',t)).\]
It is easy to see that $\overline{\gamma}$ is a constant-speed bicombing and that it is consistent if $\gamma$ and $\gamma'$ are. 
Suppose the functions $A,A'$ satisfy the monotonicity assumptions from \autoref{rem:monoton} and that $\gamma$ is $A$-convex and $\gamma'$ is $A'$-convex. Then we obtain 
\begin{align*}
& \overline{d}(\overline{\gamma}_{(x,x'),(y,y')}(t),\overline{\gamma}_{(w,w'),(z,z')}(t))\\
=~&(d(\gamma(x,y,t),\gamma(w,z,t))^2+d'(\gamma'(x',y',t),\gamma'(w',z',t))^2)^\frac{1}{2}\\
\le~& (A(t,d(x,w),d(y,z))^2+A'(t,d(x',w'),d(y',z'))^\frac{1}{2}\\
\le~& (A(t,\overline{d}((x,x'),(w,w')),\overline{d}((y,y'),(z,z')))^2\\
&+A'(t,\overline{d}((x,x'),(w,w')),\overline{d}((y,y'),(z,z')))^2)^\frac{1}{2}
\end{align*}
So $\overline{\gamma}$ is $\overline{A}$-convex for 
$\overline{A}(t,s,s')\coloneqq(A(t,s,s')^2+A'(t,s,s')^2)^\frac{1}{2}.$
\end{proof}
\begin{example}
\label{ex:prod}
Motivated by Burger and Mozes \cite{burger1997finitely} we consider groups $G$ which act properly, cocompactly and simplically on a product $T_1\times\ldots\times T_n$ of hyperbolic graphs $T_i$. Note that since the action is proper and cocompact the graphs $T_i$ are locally finite.

Let $T_i$ be $\delta_i$-hyperbolic and $v(T_i)$ be the set of vertices. By Lang \cite[Proposition 1.3]{lang} the injective hull $E(v(T_i))$ is also $\delta_i$-hyperbolic and since $v(T_i)$ is discretely geodesic every point in $E(v(T_i))$ has distance at most $\delta_i+\frac{1}{2}$ to the image of the embedding $e_i\colon v(T_i)\to E(v(T_i))$. Every discretely geodesic $\delta$-hyperbolic metric space has $(\delta+1)$-stable intervals and thus the injective hull $E(v(T_i))$ is proper and has the structure of a locally finite polyhedral complex with only finitely many isometry types of $n$-cells, isometric to injective polytopes in $l_\infty^n$ for every $n\geq 1$ by \cite[Theorem 1.1]{lang}. Combining these results we see that $E(v(T_i))$ is finite dimensional. 
When we endow the product of the injective hulls $E(v(T_i))$ with the supremum metric, then it is again injective and thus we get an isometric embedding
\[j\colon E(v(T_1)\times\ldots\times v(T_n))\to E(v(T_1))\times\ldots\times E(v(T_n)).\]
Since injective metric spaces are complete, the image is closed. From this we deduce that $E(v(T_1)\times\ldots\times v(T_n))$ is proper, finite-dimensional and within finite distance from the image of the embedding
\[e\colon v(T_1)\times\ldots\times v(T_n)\to E(v(T_1)\times\ldots\times v(T_n)).\]
By \cite[Proposition 3.8]{lang} every injective metric space admits a conical geodesic bicombing and by Descombes and Lang \cite[Theorem 1.1]{lang2} every proper metric space with a conical geodesic bicombing admits a convex geodesic bicombing. Since $E(v(T_1)\times\ldots\times v(T_n))$ is finite dimensional it has finite combinatorial dimension in the sense of Dress and thus the convex bicombing on $E(v(T_1)\times\ldots\times v(T_n))$ already is consistent and unique, i.e. it is the only convex bicombing, by \cite[Theorem 1.2]{lang2}.
The action of $G$ on $E(v(T_1)\times\ldots\times v(T_n))$ is isometric and it is proper since $E(v(T_1)\times\ldots\times v(T_n))$ is proper and within finite distance from $v(T_1)\times \ldots\times v(T_n)$. Since the convex geodesic bicombing is unique, it is equivariant with respect to the $G$-action. This shows that $G$  is a Busemann group.
\end{example}

\section{A non-CAT(0) example}
\label{sec:noncat}
Let $F$ be a hyperbolic surface, $T_1F$ the unit tangent bundle and $G\coloneqq\pi_1T_1F$. We will show that $G$ is not a CAT(0)-group but still satisfies \autoref{assumption}. 

The group $G$ fits into the non-split central extension 
\[1\to \bbZ\to G\to \pi_1F\to 1,\]
see Scott \cite[Section~4]{scott83} and Alonso and Bridson \cite[Section~8]{alonso1995semihyperbolic}. Every finite index subgroup $H$ of $\pi_1F$ is again the fundamental group of a hyperbolic surface $F'$ and the preimage of $H$ in $G$ is isomorphic to $\pi_1T_1F'$, i.e. the extension is also non-split for every finite index subgroup of $\pi_1F$. Thus $G$ is not a CAT(0)-group, see Bridson and Haefliger \cite[Theorem~II.6.12]{bh}. Note, that the Farrell--Jones conjecture for $G$ follows easily from the fact that $\pi_1(F)$ is hyperbolic and the inheritance properties of the Farrell--Jones conjecture. 

We have a covering map $T_1\mathbb{H}^2\rightarrow T_1F$ and thus both spaces have the same universal cover. Since $T_1\mathbb{H}^2$ is a fiber bundle over $\mathbb{H}^2$ with fiber $S^1$, its universal cover $X$ is a fiber bundle over $\mathbb{H}^2$ with fiber $\mathbb{R}$. Let $p$ denote the projection of $X$ to $\mathbb{H}^2$. We have that $T_1\mathbb{H}^2\cong T_1(\SL_2(\mathbb{R})/\SO_2(\mathbb{R}))\cong PSL_2(\mathbb{R})$ and thus $X\cong \widetilde{SL}_2(\bbR)$ is one of Thurston's eight three-dimensional geometries. Geodesics on $X$ have been studied by Nagy \cite{nagy1977tangent}. The upshot is that they map to paths of constant curvature in $\mathbb{H}^2$. How much they are curved depends only on the difference in the fibers. Geodesics are not unique in $X$. For this reason we will define a bicombing on $X$ that behaves much better, similar to the bicombing defined in \cite{alonso1995semihyperbolic}.

Consider a geodesic line $L\subset \mathbb{H}^2$. It is easy to check that the preimage $p^{-1}(L)$ equipped with the restriction of the left-invariant Riemannian metric on $X$ is isometric to $\mathbb{R}^2$.

Let us now define the constant speed bicombing.
Given two points $x,y\in X$ consider the unique geodesic line $L$ through $p(x),p(y)$. Let $\gamma_{x,y}$ be the geodesic between $x$ and $y$ in the plane $p^{-1}(L)$. Note that this is in general not a geodesic in $X$, since the plane is not a convex subspace. If we consider two points on such a chosen path, they still lie in the same plane $p^{-1}(L)$ and thus the chosen path between them is the restriction of the longer path. So the bicombing defined this way is consistent.

There is an interpretation for the horizontal lines, i.e. the ones which are always orthogonal to the fibers. Those are given by parallel transport, i.e. given any geodesic line $\gamma:\bbR \rightarrow \mathbb{H}^2$ and a unit tangent vector $v$ at $\gamma(0)$, then there is a unique lift $\widetilde{\gamma}:\mathbb{R}\rightarrow T_1\bbH^2$ which minimizes path length. The unit tangent vector $\widetilde{\gamma}(t)$ is given by parallel transport of $v$ along the path $\gamma|_{[0,t]}$.

Next we want to understand triangles in the upper bicombing. Parallel transport around all sides of a triangle in $\bbH^2$ does rotate a vector by the angle sum minus $\pi$, which is the area of the oriented hyperbolic triangle, see \cite[Lemma 8.4]{alonso1995semihyperbolic}. Approximating piecewise smooth curves by piecewise geodesic curves we obtain the same result for those curves, i.e. that parallel transport along a closed, piecewise smooth curve rotates a tangent vector by the oriented area that this curve encloses. 

Let $A,B\in X$ be given, let $\gamma$ be a geodesic between $A$ and $B$ and $c$ the geodesic between $p(A)$ and $p(B)$. We can identify $p^{-1}(c)$ with $[0,d(p(A),p(B))]\times\bbR$ such that $A=(0,0)$ and $B=(d(p(A),p(B)),h_B)$ for some $h_B\in\bbR$. Let $a$ be the area of the domain enclosed by $p(\gamma)$ and $c$. The isoperimetric inequality yields $a\leq l(c)+l(p(\gamma))\leq 2d(A,B)$. Let $\hat A=(d(p(A),p(B)),h)$ be the point in the fiber over $p(B)$ given by parallel transport of $A$ along $p(\gamma)$. Then $|h|=a$ and parallel transport along $p(\gamma)$ decreases the distance to $B$ in $p^{-1}(p(\gamma))$, since the Riemannian metric restricted to $p^{-1}(p(\gamma))$ is isometric to the $l^2$-metric on $[0,l(p(\gamma))]\times \bbR$. Therefore, $|h-h_B|\leq d(A,B)$ and
\[|h_B|\leq a+d(A,B)\leq 3d(A,B).\]

By construction the length of the path $\gamma_{A,B}$ is $(d(p(A),p(B))^2+h_B^2)^{1/2}\leq \sqrt{10}d(A,B)$. Let $B'\in X$ be a third point and let $h_{B'}$ be constructed analogously. 
Let $\hat B'=(d(p(A),p(B)),\hat h)$ be given by parallel transport of $B'$ along the geodesic from $p(B')$ to $p(B)$, where we still use the identification of $p^{-1}(c)$ with $[0,d(p(A),p(B))]\times\bbR$. Then as above we have $|\hat h-h_B|\leq 3d(B,B')$. Furthermore, the difference $|h_{B'}-\hat h|$ is given by the area $a$ of the hyperbolic triangle $p(A),p(B),p(B')$. 

Let $g(r)$ denote the maximal area of an (ideal) triangle in hyperbolic space $\bbH^2$ where one side has length $r$, i.e.
\[g(r)= \pi-2 \cos^{-1}(\tanh(r/2)).\]
And let $f'(x)\coloneqq(x^2+(3x+g(x))^2)^{1/2}$. Then using the calculations above we get

\begin{align*}
|l(\gamma_{A,B})-l(\gamma_{A,B'})|=~&|(d(p(A),p(B))^2+h_B^2)^{1/2}-(d(p(A),p(B'))^2+h_{B'}^2)^{1/2}|\\
=~& ((d(p(A),p(B))-d(p(A),p(B'))^2+(h_B-h_{B'})^2)^{1/2}\\
=~&(d(p(B),p(B'))^2 + (3d(B,B')+a)^2)^{1/2}\\
=~& (d(B,B')^2+(3d(B,B')+g(d(B,B')))^2)^{1/2}\\
=~&f'(d(B,B')).
\end{align*}
Setting $f(x)\coloneqq2f(x')$ we get \[|l(\gamma_{A,B})-l(\gamma_{A',B'})|\leq f'(d(A,A'))+f'(d(B,B'))\leq f(d(A,A')+d(B,B')).\]

It remains to see that the bicombing is $A$-convex for a suitable function $A$.
The length of the path from $B$ to $B'$ is given by $((d(p(B),p(B'))^2+(h_B-h_{B'}+a)^2)^{1/2}$, where $a$ denotes the area of the hyperbolic triangle $p(A),p(B),$ $p(B')$. 

Let some number $t\in [0,1]$ be given and let $C$ be the point $\gamma_{A,B}(t)$ and let $C'$ be $\gamma_{A,B'}(t)$. Using the same estimations, we get
\[l(\gamma_{C,C'})\le ((d(p(C),p(C'))^2+(th_B-th_{B'}+a_C)^2)^{1/2}\]
where $a_C$ denotes the area of the triangle $p(A),p(C),p(C')$. Since $\mathbb{H}^2$ is a CAT$(0)$-space, we get $d(p(C),p(C'))\le t \cdot d(p(B),p(B'))$.

Thus with $a(t,r)\coloneqq(t\cdot r)^2+(4t\cdot r+tg(r)+g(t\cdot r))^2)^{1/2}$ we obtain
\[d(C,C')\le l(\gamma_{C,C'})\le a(t,d(B,B')).\]
And in the same way
\[d(\gamma_{A,B'}(t),\gamma_{A',B'}(t))=d(\gamma_{B',A}(1-t),\gamma_{B',A'}(1-t))\le a(1-t,d(A,A')).\]

By the triangle inequality the bicombing is $A$-convex for the function \[A(t,x,x')\coloneqq a(t,x)+a(1-t,x')\].
\section{The flow space}
\label{sec:flow}
In this section we will define a flow space $FS(X,\gamma)$ from a bicombing $\gamma$ on $X$.
\begin{remark}
We will equip the space $\Map(\bbR,X)$ of all continuous maps from $\bbR$ into a space $X$ with the compact-open topology. If $X$ is a metric space, we can consider the closed subspace $\Lip_R(X)$ of all $R$-Lipschitz maps from $\bbR$ into $X$. The subspace topology on $\Lip_1(X)$ is induced by the metric
\[d(f,g)\coloneqq \int_\bbR \frac{d(f(t),g(t))}{2e^{|t|}}dt.\]
The $\bbR$-action $\Phi$ on $\Map(\bbR,X)$ given by $\Phi_t(f)=f(\_+t)$ restricts to $\Lip_1(X)$.
\end{remark}
\begin{defi}
\label{def:flow}
Let $G$ be a (discrete, countable) group. A \emph{flow space} for $G$ is a metric space $FS$ together with a continuous action of $G\times \bbR$, such that the action of $G$ on $FS$ is isometric and proper. We call a flow space $FS$ cocompact if the $G$ action on $FS$ is cocompact.
\end{defi}
\begin{defi}
Following the definition of a generalized geodesic from Bartels and L\"uck \cite[Definition 1.1]{flow} we call a continuous function
\[p\colon \bbR \rightarrow X\]
a \emph{trail} if there are numbers $p_-,p_+\in[-\infty,\infty]$
such that $p$ is a path of speed 1 on $[p_-,p_+]$ and locally constant on the complement.
\end{defi}

If the trail $p$ is not constant, than the numbers $p_-,p_+$ are uniquely determined by $p$.

\begin{rem}
In general the space of all trails need not be a closed subspace of $\Lip_1(X)$, i.e. the limit of a sequence of paths of speed one could have speed less than one. In \autoref{lem:res} we show that this problem does not arise when we consider the space of trails coming from a bicombing as in the next definition if the bicombing is consistent.

Another problem that arises in the general setting which is not present in the CAT(0)-case is that the space of all trails is in general not finite dimensional even if the space $X$ is. We will now show how to construct a finite-dimensional flow space $FS(X,\gamma)$ starting from the conditions in \autoref{assumption}. The main ingredient for the constructed flow space to be finite-dimensional is again the consistency of the bicombing.
\end{rem}

\begin{definition}\label{def:flowspace} Given a constant speed bicombing $\gamma:X\times X\times [0,1]\rightarrow X$, we can consider the space consisting of all trails $c_{x,y}$ of the form
\[c_{x,y}\coloneqq \begin{cases} x& t\le 0 \\ \gamma_{x,y}(t/l(\gamma_{x,y})) & t\in [0,l(\gamma_{x,y})] \\ y & t\ge l(\gamma_{x,y})\end{cases},\]
where $x,y$ are some points in $X$. The flow space $FS(X,\gamma)$ is the closure of $A(X,\gamma)\coloneqq\{\Phi_tc_{x,y}\mid t\in \IR,x,y\in X\}$ in the space $\Map(\bbR,X)$.
\end{definition}
Note that $A(X,\gamma)$ is contained in the space $\Lip_1(X)$ of all $1$-Lipschitz maps $\bbR \to X$ and thus its closure $FS(X,\gamma)$ is also contained in $\Lip_1(X)$. We will restrict its metric to $FS(X,\gamma)$.

In general the map that assigns to the pair $(x,y)$ the path above need not be continuous, as the path length can be discontinuous viewed as a map from $\Map([0,1],X)$ with the compact open topology to $\bbR$. The existence of a continuous function $f\colon[0,\infty)\to[0,\infty)$ with  $|l(\gamma_{x,y})-l(\gamma_{x',y'})|\leq f(d(x,x')+d(y,y'))$ for all $x,x',y,y'\in X$ and $f(0)=0$ as in \autoref{assumption} ensures continuity.

If $X$ has an (isometric) $G$-action and $\gamma$ is equivariant, then $FS(X,\gamma)$ also inherits an (isometric) $G$-action.

\begin{lemma}~
\label{lem:1.4}\begin{enumerate} \item 
Let $c,d\in \Lip_1(X)$ be given. For all $t_0 \in \bbR$
\[d_X(c(t_0), d(t_0))\le e^{|t_0|}\cdot d_{\Lip_1(X)}(c,d)+2.\]

\item For $c,d\in \Lip_1(X)$, $\sigma,\tau \in \bbR$ we have
\[d_{\Lip_1(X)}(\Phi_\tau(c),\Phi_\sigma(d))\le e^{|\tau|}d_{\Lip_1(X)}(c,d)+|\sigma-\tau|.\]
\end{enumerate}
\end{lemma}
\begin{proof}The lemma is essentially \cite[Lemma~1.3 and Lemma 1.4]{flow} and the proof there carries over to the setting of $1$-Lipschitz maps.
\end{proof}

\begin{lemma}
\label{lem:lipproper}
If $X$ is a proper metric space, then the space $\Lip_1(X)$ with the metric given above is also proper.
\end{lemma}
\begin{proof}
Let $f\in \Lip_1(X), R\in\bbR$ be given. By \autoref{lem:1.4} for all $t\in\bbR$ we have $\{f'(t)\mid f'\in B_R(f)\}\subseteq B_{e^{|t|}R+2}(f(t))$ and since $X$ is proper, its closure is compact. Furthermore, $B_R(f)$ is a subspace of $\Lip_1(X)$ and thus equicontinuous. By a version of the Arzela-Ascoli theorem (\cite[Theorem 47.1]{munkres}) the closure of $B_R(f)$ is compact.
\end{proof}

\begin{corollary}
The space $FS(X,\gamma)$ is proper, since it is a closed subspace of a proper metric space and the metric is just the restriction.
\end{corollary}

\begin{corollary}\label{cor:evproper}
The evaluation map 
\[ev_0\colon FS(X,\gamma)\rightarrow X,\qquad f\mapsto f(0)\]
is proper.
\end{corollary}
\begin{proof}
By the definition of the metric on $FS$, the evaluation map $ev_0$ is continuous.
The preimage of an $R$-ball around a point $x$ consists of 1- Lipschitz maps $f$ with $f(0)\in B_R(x)$. For any two such maps we have
\[d_{FS}(f,f')\le \int_\bbR \frac{2R+|2t|}{2e^{|t|}}dt=:R'<\infty.\]
Thus the preimage is contained in a ball of radius $R'$ around any of its points. This ball is compact by the last lemma. 
\end{proof}

\begin{corollary} \label{cor:cocompactac}
If $X$ has a cocompact, isometric $G$-action and the bicombing $\gamma$ is equivariant, then the induced $G$-action on $FS(X,\gamma)$ is also cocompact.
\end{corollary}
\begin{proof}
Pick a compact set $K\subset X$ with $GK=X$. Then $ev_0^{-1}(K)$ is also compact and $G\cdot ev_0^{-1}(K)=ev_0^{-1}(X)=FS(X,\gamma).$
\end{proof}

\begin{lemma} \label{lem:properac}
If $X$ has a proper, isometric $G$-action and the bicombing $\gamma$ is equivariant, then the induced $G$-action on $FS(X,\gamma)$ is also proper.
\end{lemma}
\begin{proof}
We have an equivariant map $ev_0:FS(X,\gamma)\rightarrow X$ to a space with a proper $G$- action. For any compact subset $K\subset FS(X,\gamma)$, we have that 
\[\{g\in G \mid gK\cap K\neq \emptyset\}\subset \{g\in G\mid g\cdot ev_0(K)\cap ev_0(K)\neq \emptyset\}\]
and the latter is finite. Thus the $G$-action on $FS(X)$ is also proper. 
\end{proof}

\begin{remark}~\begin{enumerate}
\item Let $\Lip_1(X,[a,b])$ denote the subset of $\Lip_1(X)$ consisting of all maps that are locally constant on the complement of the interval $[a,b]$. We have a retraction 
\[\res_{[a,b]}:\Lip_1(X)\rightarrow \Lip_1(X,[a,b]),\qquad f\mapsto \left( t\mapsto \begin{cases} f(a)& t\le a\\ f(t) & t\in [a,b]\\ f(b) & t\ge b \end{cases} \right).\]
The continuity of this map can be easily verified. Since
\[\Lip_1(X,[a,b])=\{f\in \Lip_1(X)\mid\res_{[a,b]}(f)=f\}\]
is a retract of a Hausdorff space, $\Lip_1(X,[a,b])$ is a closed subspace.
\item Let $A(X,\gamma,[a,b])\coloneqq A(X,\gamma)\cap\Lip_1(X,[a,b])$. If $\gamma$ is consistent, we can restrict the upper retraction to
\[\res_{[a,b]}:A(X,\gamma)\rightarrow A(X,\gamma,[a,b]).\]
\end{enumerate}
\end{remark}

\begin{lemma}
\label{lem:resclosed}
The subspace $A(X,\gamma,[a,b])\subseteq\Lip_1(X)$ is closed.
\end{lemma}
\begin{proof}
Let $c_n\in A(X,\gamma,[a,b])$ converge to $c\in\Lip_1(X)$. Then $c_n(a)$ converges to $c(a)$ and $c_n(b)$ to $c(b)$. There exist $t_n\in[a,b]$ such that $c_n=\Phi_{-t_n}c_{c_n(a),c_n(b)}$. Consider a subsequence such that $t_n$ converges to $t\in[a,b]$. Since both the flow $\Phi$ and the bicombing $c$ are continuous, so is the map
\[X\times X\times \bbR\to \Lip_1(X),\quad (x,y,t)\mapsto \Phi_tc_{x,y}.\]
Hence $c_n$ converges to $\Phi_{-t}c_{c(a),c(b)}$, i.e.
$c=\Phi_{-t}c_{c(a),c(b)}\in A(X,\gamma)$.
Since the space $\Lip_1(X,[a,b])$ is a closed subspace of $\Lip_1(X)$, we have \[c\in A(X,\gamma)\cap  \Lip_1(X,[a,b])=A(X,\gamma,[a,b]).\qedhere\]
\end{proof}

\begin{lemma}\label{lem:res} Let $\gamma$ be a consistent bicombing on $X$. Then the space $FS(X,\gamma)$ consists of all trails $w$ such that for any interval $[a,b]$ of finite length
\[\res_{[a,b]}w\in A(X,\gamma,[a,b]),\]
i.e. $\res_{[a,b]}w=\Phi_{-\max\{a,w_-\}}c_{w(a),w(b)}$.
\end{lemma}
\begin{proof}
Let $w\in \Lip_1(X)$ be a trail such that for all intervals of finite length $[a,b]$ we have $\res_{[a,b]}w\in A(X,\gamma)$. Then the sequence $(\res_{[-n,n]}w)_{n\in\bbN}$ in $A(X,\gamma)$ converges to $w$ and thus $w\in FS(X,\gamma)$.

For any interval $[a,b]$ of finite length $A(X,\gamma)$ is mapped to $A(X,\gamma,[a,b])$ under $\res_{[a,b]}$ since $\gamma$ is consistent. By \autoref{lem:resclosed} $A(X,\gamma,[a,b])\subseteq \Lip_1(X)$ is closed and thus also $FS(X,\gamma)$ is mapped to $A(X,\gamma,[a,b])$ under $\res_{[a,b]}$. 

It remains to show that any $w\in FS(X,\gamma)$ is a trail. Let $w_-,w_+$ be maximal such that $w$ is parametrised by path length on $[-w_-,w_+]$. If $w_-,w_+=\infty$, the element $w$ is a trail. Thus, let us assume $w_+<\infty$. For $n>w_+$ the trail $\res_{[-n,n]}w$ is only parametrised by path length on $[-\max\{n,w_-\},w_+]$ and thus has to be constant on $[w_+,n]$. This shows that $w$ is constant on $[w_+,\infty)$ and by the same argument on $(-\infty,-w_-]$, i.e. $w$ is a trail.
\end{proof}

\begin{lemma}
\label{lem:properclosed}
Proper maps between locally compact Hausdorff spaces are closed.
\end{lemma}
\begin{proof}
Let $f\colon X\to Y$ be proper and $Y$ locally compact and Hausdorff. Let $C\subseteq X$ be closed. Every point $y\in Y\setminus f(C)$ has an open neighborhood $U$ with compact closure. Since $f$ is proper also $f^{-1}(\overline{U})$ and $K\coloneqq f^{-1}(\overline{U})\cap C$ are compact. Thus $f(K)$ is compact and hence closed. Now $U\setminus f(K)$ is an open neighborhood of $y$ disjoint from $f(C)$. 
\end{proof}

\begin{lemma} Let $\gamma$ be a consistent bicombing on $X$. The map $\lim_{n\in\bbN}\res_{[-n,n]}:FS(X,\gamma) \rightarrow  \lim_{n\in \bbN} A(X,\gamma,[-n,n])$ is a homeomorphism.
\end{lemma}
\begin{proof}
\autoref{lem:res} already shows that it is a continuous bijection. We have to show that it is closed or by \autoref{lem:properclosed} that it is proper. Let $K$ be any compact subset of $\lim_{n\in \bbN} A(X,\gamma,[-n,n])$. We have to show that its preimage is compact. By enlarging $K$, we may assume without loss of generality that it is the preimage of some compact subset $K'\subset A(X,\gamma,[0,0])$. Note that
\[FS(X,\gamma)\rightarrow A(X,\gamma,[0,0]) = X\] 
is the evaluation at $0$ and thus a proper map by \autoref{cor:evproper}. 
\end{proof}

\begin{lemma}\label{lem:A-findim}
Let $\gamma$ be a bicombing on $X$ as in \autoref{assumption}. The spaces $A(X,\gamma)$ and $A(X,\gamma,[a,b])$ have covering dimension at most $2\dim X+1$.
\end{lemma}
\begin{proof}
The space $A(X,\gamma)$ is the union of the constant trails $A(X,\gamma)^\bbR$ and the non-constant ones. The evaluation map is an isometry of $A(X,\gamma)^\bbR$  to $X$ and $\dim A(X,\gamma)^\bbR=\dim X$. Let us examine the second part. Consider the continuous map
\[((X\times X)\setminus \Delta(X))\times \bbR \rightarrow A(X,\gamma)\setminus A(X,\gamma)^\bbR,\quad (x,y,t)\mapsto \Phi_tc_{x,y},\]
where $\Delta\colon X\to X\times X$ denotes the diagonal map. The map is bijective by definition of $A(X,\gamma)$. 
We will now show that the map is also proper. For this it suffices that every point in the image has a small neighborhood such that the closure of the preimage is compact. Given $c\coloneqq\Phi_tc_{x,y}\in A(X,\gamma)\setminus A(X,\gamma)^\bbR$. Since $\Lip_1(X)^\bbR\subseteq\Lip_1(X)$ is closed, we have $\epsilon\coloneqq d(c,\Lip_1(X)^\bbR)>0$. Let $r>0$ be such that $\int_{-\infty}^{-r}\tfrac{|t|}{2e^{|t|}}dt<\epsilon/2$. Define $R\coloneqq\max\{1,-\frac{\log(\epsilon/2)}{2}, r, -c_-, c_+, f(d(x,y)+6)\}$. For $c'\in B_{e^{-2R}}(c)$ and $t\in[-2R,-R]$ we have by \autoref{lem:1.4}
\[d(x,c'(t))=d(c(t),c'(t))\leq e^{|t|-2R}+2\leq 3\]
and for $t\in[R,2R]$
\[d(y,c'(t))=d(c(t),c'(t))\leq e^{|t|-2R}+2\leq 3.\]
Thus $d(c'(-2R),c'(2R))\leq d(x,y)+6$ and $l(\res_{[-2R,2R]}c')\leq f(d(x,y)+6)\leq R$.

Claim: The trail $c'$ is constant on $(-\infty,-2R]$.

Since $c'$ is parametrised by path length it otherwise has to be constant on $[-R,2R]$ by the above and since it is a trail also on $[-R,\infty)$. Then the distance to the constant map to $c'(-R)$ is
\[\int_{-\infty}^\infty\tfrac{d(c'(t),c'(-R))}{2e^{|t|}}dt=\int_{-\infty}^{-R}\tfrac{d(c'(t),c'(-R))}{2e^{|t|}}dt\leq \int_{-\infty}^{-R}\tfrac{|t|}{2e^{|t|}}dt<\epsilon/2.\]
Thus $d(c,\Lip_1(X)^\bbR)\leq d(c,c')+d(c',\Lip_1(X)^\bbR)<e^{-2R}+\epsilon/2\leq \epsilon$ which is a contradiction to the definition of $\epsilon$. 

By the same argument it follows that $c'$ is constant on $[2R,\infty)$ and there exists $t\in[-2R,2R]$ with $c'=\Phi_tc_{c(-2R),c(2R)}$. In particular, the pre-image of $B_{1/2R}(c)$ is contained in $K\coloneqq B_3(x)\times B_3(y)\times [-2R,2R]$ and the closure of $K$ is compact.

It follows that $A(X,\gamma)\setminus A(X,\gamma)^\bbR$ is homeomorphic to a subset of $X\times X\times \bbR$ and by the subspace theorem \cite[Theorem~1.1.2]{engelking1978dimension} it has dimension at most $2\dim X+1$. Since $A(X,\gamma)=A(X,\gamma)^\bbR\cup(A(X,\gamma)\setminus A(X,\gamma)^\bbR)$ and $A(X,\gamma)^\bbR\subseteq A(X,\gamma)$ is closed, this implies $\dim A(X,\gamma)\leq 2\dim X+1$ by \cite[Corollary~1.5.5]{engelking1978dimension}. Using the subspace theorem again, this also holds for $A(X,\gamma,[a,b])$.
\end{proof}
\begin{corollary}
\label{cor:findim}
The space $FS(X,\gamma)$ has covering dimension at most $2\dim X+1$.
\end{corollary}
\begin{proof}
Since $FS(X,\gamma)$ is a proper metric space, it can be written as a countable union of closed balls $B_1(x_i)$ for some sequence of points $x_i\in FS(X,\gamma)$. By the sum theorem
\cite[Theorem~1.5.3]{engelking1978dimension} it suffices to show that every such ball has dimension at most $2\dim X+1$. Let $f_n:FS(X,\gamma)\rightarrow A(X,\gamma,[-n,n])$ denote the restriction map. We get
\[B_1(x_i)=\lim_n f_n(B_1(x_i)).\]
Let $\calu$ be any open cover of $B_1(x_i)$. We have to find an $2\dim X+1$-dimensional refinement of $\calu$.  A basis for the topology of an inverse limit is given by the family $\{f_n^{-1}(U)\mid n\in \bbN, U\subset f_n(B_1(x_i)) \mbox{ open}\}$. Thus we can assume without loss of generality that every open set in $\calu$ is of this form. Since $B_1(x_i)$ is compact, we can pick a finite subcover. 
Thus there is some $N$ such that every set of this subcover is a pullback of an open set in $f_N(B_1(x_i))$. Thus the whole 
cover is a pullback of some cover $\calu'$ of $f_N(B_1(x_i))$.
Since $f_N(B_1(x_i))$ is a subspace of $A(X,\gamma,[-N,N])$, we can use the subspace theorem \cite[Theorem~1.2.2]{engelking1978dimension}
and \autoref{lem:A-findim} to find a $2\dim X+1$-dimensional subcover $\calu''$ of $\calu'$. Then $f_N^{-1}(\calu'')$ is the desired refinement. 
\end{proof}
\section{Contracting transfers}
\label{sec:trans}
In this section $(X,\gamma), A$ and $f$ will always be as in \autoref{assumption} and we assume the monotonicity assumptions from \autoref{rem:monoton}. Furthermore, $FS(X,\gamma)$ will be the flow space as defined in \autoref{def:flowspace}, where the trails induced by $\gamma$ are denoted by $c_{x,y}$.

\begin{definition}[{\cite[Definition~2.1,Definition~2.3]{wegnercat}}] A \emph{strong homotopy action} of a group $G$ on a topological space $X$
is a continuous map
\[\Psi: \coprod_{j=0}^\infty (G\times [0,1])^j\times G\times X \rightarrow X\]
with the following properties:
\begin{enumerate}
\item $\Psi(\ldots,g_l,0,g_{l-1},\ldots) = \Psi(\ldots,g_l,\Psi(g_{l-1},\ldots))$,
\item $\Psi(\ldots,g_l,1,g_{l-1},\ldots) = \Psi(\ldots,g_l \cdot g_{l-1},\ldots)$,
\item $\Psi(e,t_j,g_{j-1},\ldots)=\Psi(g_{j-1},\ldots)$,
\item $\Psi(\ldots,t_l,e,t_{l-1},\ldots) = \Psi(\ldots,t_l \cdot t_{l-1},\ldots)$,
\item $\Psi(\ldots,t_1,e,x) = \Psi(\ldots,x)$,
\item $\Psi(e,x) = x$.
\end{enumerate}
For a subset $S\subseteq G$ containing $e$,$g\in G$ and a $k\in \bbN$ define
\[F_g(\Psi,S,k)\coloneqq  \{\Psi(g_k,t_k,\ldots,g_0,?):X\rightarrow X\mid g_i\in S,t_i\in [0,1],g_k \ldots g_0=g\}.\]
For $(g,x)\in G\times X$ we define $S^0_{\Psi,S,k}(g,x)$ as $\{(g,x)\}$, $S^1_{\Psi,S,k}(g,x)\subseteq G\times X$ as the subset consisting of all $(h,y)\in G\times X$ with the following property: There are $a,b\in S$, $f\in F_a(\Psi,S,k),f'\in F_b(\Psi,S,k)$ such that $f(x)=f'(y)$ and $h=ga^{-1}b$.

For $n\ge 2$ define inductively $S^n_{\Psi,S,k}(g,x)=\bigcup_{(h,y)\in S^{n-1}_{\Psi,S,k}(g,x)} S^1_{\Psi,S,k}(h,y)$.
\end{definition}
In case of a strict group action, the sets $F_g(\Psi,S,k)$ contain only the map $X\rightarrow X$ that is given by multiplication with $g$ and the sets $S^n(\Psi,S,k)(g,x)$ are analogs of the sets $B_{2nk}(e)\cdot (g,x)$, where $B_{2nk}(e)$ denotes the ball around the neutral element $e\in G$ with respect to the word metric with generating set $S$ and the $G$-action on $G\times X$ given by $s\cdot (g,x) = (gs^{-1},sx)$.

\begin{defi}[{\cite[Definition 0.2]{flow}}]
Let $X$ be a metric space and $N\in\bbN$. We say that $X$ is controlled $N$-dominated if for every $\epsilon>0$ there is a finite CW-complex $K$ of dimension at most $N$, maps $i\colon X\to K, p\colon K\to X$ and a homotopy $H\colon X\times [0,1]\to X$ between $p\circ i$ and $\id_X$ such that for every $x\in X$ the diameter of $\{H(x,t)\mid t\in [0,1]\}$ is at most $\epsilon$.
\end{defi}
In our situation we will actually have $p\circ i=\id_X$.

\begin{definition}[{\cite[Definition~9.6]{covers}}] \label{def:sct}A flow space FS for a group $G$ admits \emph{strong contracting transfers} if there is an $N\in\bbN$ such that for every finite subset $S$ of $G$ and every $k\in \bbN$ there exists $\beta>0$ such that the following holds. For every $\delta>0$ there is 
\begin{enumerate}
\item a number $T>0$;
\item a contractible, compact, controlled $N$-dominated space $X$;
\item a strong homotopy action $\Psi$ on $X$;
\item a $G$-equivariant map $\iota: G\times X\rightarrow FS$ (where the $G$-action on $G\times X$ is given by $g\cdot (g',x)=(gg',x)$) such that the following holds:
\item[($\ast$)] \label{def:sct5} for every $(g,x)\in G\times X, s\in S,f\in F_s(\Psi,S,k)$ there is a $\tau \in [-\beta,\beta]$ such that
$d_{FS}(\Phi_T\iota(g,x),\Phi_{T+\tau}\iota(gs^{-1},f(x)))\le \delta$.
\end{enumerate}
\end{definition}
The goal of this section is to prove the following:
\begin{prop}
\label{prop:transfer}
If $(X,\gamma)$ is as in \autoref{assumption}, then the flow space $FS(X,\gamma)$ admits strong contracting transfers.
\end{prop}
\begin{defi}
We define $P_r(x)\coloneqq\{y\in X\mid l(c_{x,y})\leq r\}$.
\end{defi}
Before we can prove \autoref{prop:transfer} we need the following technical statements, which are the analogs of \cite[Lemma 3.6 and Proposition 3.5]{flow} in our setting. They will be important to estimate distances in the flow space. The next lemma shows that in a triangle of trails the two trails run close to each other for a long time. This is used in the following proposition to prove the existence of a flow time $T$ moving these trails close together in the flow space.

\begin{lemma}
\label{lemma:3.6}
Let $r',L,\beta>0,r''>f(\beta)$ and $x_1,x_2\in X$ with $d(x_1,x_2)\leq\beta$ and $\tfrac{r''}{r''+2r'+f(\beta)+L}\geq 2/3$ be given. Let $r\geq r''+2r'+f(\beta), T\coloneqq r-r'-f(\beta), x\in P_{r+L}(x_1)$ and $\tau\coloneqq l(c_{x_2,x})-l(c_{x_1,x})$. Then for all $t\in [T-r',T+r']$ we have
\[d_X(c_{x_1,x}(t),c_{x_2,x}(t+\tau))\leq A\left(\tfrac{r''}{r''+2r'+f(\beta)+L},\beta,0\right)+\left(\tfrac{f(\beta)(2r'+f(\beta)+L)}{r''}\right).\]
\end{lemma}
\begin{proof}
Let $t\in[T-r',T+r']$. From $T-r'\geq r''>f(\beta)$ and $|\tau|\leq f(\beta)$ we conclude $t,t+\tau>0$. If $t\geq l(c_{x_1,x})$, then $c_{x_1,x}(t)=x=c_{x_2,x}(t+\tau)$ and the statement of the lemma follows in this case. Therefore, we can assume $0<t<l(c_{x_1,x})$ and thus $0<t+\tau<l(c_{x_2,x})$. 
We obtain
\begin{align*}
&d_X(c_{x_1,x}(t),c_{x_2,x}(t+\tau))\\
\leq~& d_X(c_{x_1,x}(t),c_{x_2,x}(\tfrac{t}{l(c_{x_1,x})}l(c_{x_2,x})))+|t+\tau-\tfrac{t}{l(c_{x_1,x})}l(c_{x_2,x})|\\
\leq~& A(\tfrac{t}{l(c_{x_1,x})},\beta,0)+|t+\tau-\tfrac{t}{l(c_{x_1,x})}(l(c_{x_1,x})+\tau)|\\
=~&A(\tfrac{t}{l(c_{x_1,x})},\beta,0)+|\tfrac{l(c_{x_1,x})-t}{l(c_{x_1,x})}\tau|\\
\leq~&A(\tfrac{t}{l(c_{x_1,x})},\beta,0)+(f(\beta)\tfrac{l(c_{x_1,x})-t}{l(c_{x_1,x})})
\end{align*}
where the second inequality follows from the $A$-convexity of the bicombing.\\
Let $a:=r-r''-2r'-f(\beta)>0$. Since $l(c_{x_1,x})\leq r+L$ and $t\geq r-2r'-f(\beta)=r''+a$, we have
\[\tfrac{t}{l(c_{x_1,x})}\geq \tfrac{r''+a}{r''+2r'+f(\beta)+L+a}\geq \tfrac{r''}{r''+2r'+f(\beta)+L}\]
and the lemma follows from the monotonicity assumption in \autoref{rem:monoton} on $A$.
\end{proof}
\begin{prop}
\label{prop:3.5}
Let $\beta,L>0$ be given. For all $\delta>0$ there are $T,r>0$ such that for all $x_1,x_2\in X$ with $d(x_1,x_2)\leq\beta,x\in P_{r+L}(x_1)$ there is $\tau\in[-f(\beta),f(\beta)]$ such that
\[d_{FS}(\Phi_T(c_{x_1,c_{x_1,x}(r)}),\Phi_{T+\tau}(c_{x_2,c_{x_2,x}(r)}))\leq\delta.\]
\end{prop}
\begin{proof}
Let $\beta,L,\delta$ be given. Pick $r'>1,r''>f(\beta),1>\delta'>0$ such that
\[\int_{-\infty}^{-r'}\frac{1+|t|}{e^{|t|}}dt\leq\frac{\delta}{3},\qquad\int_{-r'}^{r'}\frac{\delta'}{e^{|t|}}dt\leq\frac{\delta}{3},\qquad \tfrac{r''}{r''+2r'+f(\beta)+L}\geq 2/3\]
and
\[A(\tfrac{r''}{r''+2r'+f(\beta)+L},\beta,0)+(\tfrac{f(\beta)(2r'+f(\beta)+L)}{r''})\leq \delta'.\]
Define $r\coloneqq2r'+r''+f(\beta)$ and $T\coloneqq r-r'-f(\beta)$. Let $x_1,x_2\in X$ with $d_X(x_1,x_2)\leq\beta$ and $x\in P_{r+L}(x_1)$ be given. Set $\tau\coloneqq l(c_{x_2,x})-l(c_{x_1,x})$. Then $|\tau|\leq f(\beta)$. By \autoref{lemma:3.6} for all $t\in[-r',r']$ we have
\[
d(c_{x_1,x}(T+t),c_{x_2,x}(T+t+\tau))\leq A(\tfrac{r''}{r''+2r'+f(\beta)+L},\beta,0)+(\tfrac{f(\beta)(2r'+f(\beta)+L)}{r''})<\delta'.\]

Furthermore, for $t\in[-r',r']$ we have
\[0<r''\leq T+t=r'+r''+t\leq r,\quad 0<r''+\tau\leq T+t+\tau\leq r\]
and thus by consistency of the bicombing we obtain
\[c_{x_1,x}(T+t)=c_{x_1,c_{x_1,x}(r)}(T+t),\quad c_{x_2,x}(T+t+\tau)=c_{x_2,c_{x_2,x}(r)}(T+t+\tau).\]
This implies
\begin{align*}
&d_{FS}(\Phi_T(c_{x_1,c_{x_1,x}(r)}),\Phi_{T+\tau}(c_{x_2,c_{x_2,x}(r)}))\\
=~&\int_{-\infty}^\infty\tfrac{d(c_{x_1,c_{x_1,x}(r)}(T+t),c_{x_2,c_{x_2,x}(r)}(T+t+\tau))}{2e^{|t|}}dt\\
\leq~&\int_{-\infty}^{-r'}\tfrac{2|t|+2\delta'}{2e^{|t|}}dt+\int_{-r'}^{r'}\tfrac{2\delta'}{2e^{|t|}}dt+\int_{r'}^{\infty}\tfrac{2\delta'+2|t|}{2e^{|t|}}dt\\
\leq~& \int_{-\infty}^{-r'}\tfrac{1+|t|}{e^{|t|}}dt+\int_{-r'}^{r'}\tfrac{\delta'}{e^{|t|}}dt+\int_{r'}^{\infty}\tfrac{1+t}{e^{|t|}}dt\\
\leq~& \tfrac{\delta}{3}+\tfrac{\delta}{3}+\tfrac{\delta}{3}=\delta. \qedhere
\end{align*}
\end{proof}

\begin{proof}[Proof of \autoref{prop:transfer}]
Let a basepoint $x_0\in X$ be given. We will use the the space $P_R(x_0)$, for $R=R(S, k, \delta)$ big enough, to verify the existence of contracting transfers.

Let a finite subset $S\subseteq G$, $\delta>0$ and $k\in\bbN$ be given. Let $N\coloneqq  2\dim(X)+1$, $\beta'\coloneqq \max\{d(gx_0,hx_0)\mid g,h\in S\}$ and $\beta\coloneqq (k+1)f(\beta')$. By \autoref{prop:3.5} there exist $T,R>0$ such that for all $x_1,x_2\in X$ with $d(x_1,x_2)\leq\beta'$ and $x\in P_{R+\beta}(x_1)$ there is $\tau\in[-f(\beta'),f(\beta')]$ such that
\begin{equation} d(\Phi_T(c_{x_1,c_{x_1,x}(R)}),\Phi_{T+\tau}(c_{x_2,c_{x_2,x}(R)}))\leq \frac{\delta}{e^{\beta}(k+1)}.\label{eq:?}\end{equation}
Define a deformation retraction of $X$ onto $P_R(x_0)$ by
\[H\colon X\times [0,1]\to X,~\quad c_{x_0,x}((1-t)(R-l(c_{x_0,x}))+l(c_{x_0,x})).\]
Since $\gamma$ is consistent, we have that $H_t\circ H_{t'}=H_{tt'}$ and by \cite[Remark 2.2]{wegnercat} we can define a strong homotopy action $\Psi\coloneqq H_0\circ \Omega$ on $P_R(x_0)$ with
\[\Omega\colon \coprod_{j=0}^\infty(G\times[0,1])^j\times G\times \overline{P_R(x_0)}\to X,\]
\[\Omega(g,x)\coloneqq gx,\quad \Omega(g_j,t_j,g_{j-1},\ldots)\coloneqq g_jH_{t_j}(\Omega(g_{j-1},\ldots)).\]
The space $P_R(x_0)$ is a contractible, compact, controlled $2\dim(X)+1$-dominated metric space by \autoref{lem:6.2}. We obtain a $G$-equivariant map
\[\iota: G\times P_R(x_0) \rightarrow FS(X,\gamma),\qquad (g,x)\mapsto c_{gx_0,gx}.\]

It remains to prove \autoref{def:sct}~($\ast$). 
Now let $z\in P_R(x_0)$ and $a\in G$ be given. By induction on $m=0,\ldots,k$ we want to show that for $f\in F_a(\Psi,S,m)$ there exists $\tau\in[-(m+1)f(\beta'),(m+1)f(\beta')]$ such that

\[d(\Phi_T\iota(e,z),\Phi_{T+\tau}\iota(a^{-1},f(z)))\leq \frac{(m+1)\delta}{k+1}.\]

For $m=0$ we have $f(z)=\Psi(a,z)=H_0(az)=c_{az,x_0}(l(c_{az,x_0})-R)$. Thus by \eqref{eq:?} there is $\tau\in[-f(\beta'),f(\beta')]$ with
\begin{align*}
d(\Phi_T\iota(e,z),\Phi_{T+\tau}\iota(a^{-1},f(z)))=~&
d(\Phi_Tc_{x_0,z},\Phi_{T+\tau}c_{a^{-1}x_0,c_{a^{-1}x_0,z}(R)})\\
\leq~&\frac{\delta}{e^\beta(k+1)}\leq~\frac{\delta}{k+1} .\end{align*}
Now let us consider the induction step. By definition of $F_{a}(\Psi,S,m)$ there are $g_m,\ldots,g_0\in S$, $t_m,\ldots,t_1\in [0,1]$ with $a=g_m\ldots g_0$ and 
\[f(z)=H_0(\Omega(g_m,t_m,\ldots, g_0,z))=H_0(g_m H_{t_m}(\Omega(g_{m-1},t_{m-1},\ldots, g_0,z))).\]
Define $\overline{f}\coloneqq  H_{t_m}\circ \Omega(g_{m-1},t_{m-1},\ldots, g_0,\_),f'\coloneqq H_0\circ \overline{f}\in F_{g_m^{-1}a}(\Psi, S, m-1)$.
By induction assumption, there is a $\tau_1\in[-mf(\beta'),mf(\beta')]$ with
\[d(\Phi_T\iota(e,z),\Phi_{T+\tau_1}\iota(a^{-1}g_m,f'(z)))\leq \frac{m\delta}{(k+1)}.\]
By \eqref{eq:?} there is $\tau_2\in [-f(\beta'),f(\beta')]$ with 

\begin{align*}&d(\Phi_{T}\iota(e,f'(z)),\Phi_{T+\tau_2}\iota(g_m^{-1},f(z)))\\
=~&d(\Phi_T(c_{x_0,H_0(\overline{f}(z))}),\Phi_{T+\tau_2}(c_{g_m^{-1}x_0,g_m^{-1}c_{x_0,g_m\overline{f}(z)}(R)}))\\
=~&d(\Phi_T(c_{x_0,c_{x_0,\overline{f}(z)}(R)}),\Phi_{T+\tau_2}(c_{g_m^{-1}x_0,c_{g_{m}^{-1}x_0,\overline{f}(z)}(R)}))\\\leq~& \frac{\delta}{e^\beta(k+1)}.\end{align*}
Therefore, by \autoref{lem:1.4}
we obtain
\[d(\Phi_{T+\tau_1}\iota(e,f'(z)),\Phi_{T+\tau_1+\tau_2}\iota(g_m^{-1},f(z)))\leq \frac{\delta}{k+1}.\]
Let $\tau\coloneqq \tau_1+\tau_2$. Now we can use the triangle inequality:
\begin{align*}
&d(\Phi_T\iota(e,z),\Phi_{T+\tau}\iota(a^{-1},f(z)))\\
\le~&d(\Phi_T\iota(e,z),\Phi_{T+\tau_1}\iota(a^{-1}g_m,f'(z)))\\
&+d(\Phi_{T+\tau_1}\iota(a^{-1}g_m,f'(z)),\Phi_{T+\tau_1+\tau_2}\iota(a^{-1},f(z)))\\
\le~& \frac{m\delta}{k+1}+\frac{\delta}{k+1}=\frac{(m+1)\delta}{k+1}.\qedhere
\end{align*}
\end{proof}
\begin{lemma}
For every $x\in X$, $R>0$ the space $P_R(x)$ is compact, contractible and a neighborhood of $x$.
\end{lemma}
\begin{proof}
Since $X$ is proper, the closed ball $\overline {B_{f(R)}(x_0)}$ is compact. Therefore, also the closed subspace $P_R(x_0)\subseteq \overline{B_{f(R)}(x_0)}$ is compact. The space $P_R(x_0)$ inherits from $X$ a metric and is contractible with a contraction given by the chosen paths from $x_0$. Those paths stay in $P_R(x_0)$ because the bicombing is consistent.

There exists $r$ with $f(r)\leq R$ and thus $B_r(x)\subseteq P_{f(r)}(x)\subseteq P_R(x)$ and hence $P_R(x)$ is a neighborhood of $x$.
\end{proof}
\begin{lemma}
\label{lem:6.1}
The space $X$ is a Euclidean neighborhood retract, i.e., there is a
natural number $N$, a closed subset $A\subseteq\bbR^N$, an open neighborhood $U$ of $A$ in $\bbR^N$ and a map $r\colon U\to A$ such that $r|_A = \id_A$ and $X$ is homeomorphic to $A$. The number $N$ can be chosen to be $2\dim(X)+1$.
\end{lemma}
\begin{proof}
Since $X$ is proper as a metric space, it is locally compact and has a countable basis for its topology. To see the latter take some $x\in X$ and for every $n\in\bbN$ choose a finite cover $\{B_{1/n}(y_i)\}_{y_i\in X}$ of $B_n(x)$. This gives a countable basis for the topology.
Obviously $X$ is Hausdorff. By assumption $\dim(X)<\infty$. We conclude from \cite[Exercise 6 in Chapter 50 on page 315]{munkres} that $X$ is homeomorphic to a closed subset $A$ of $\bbR^N$ for $N = 2\dim(X) + 1$. 
For every $x\in X, \epsilon >0$ the subspace $P_\epsilon(x)\subseteq B_\epsilon(x)$ is a contractible neighborhood of $x$, hence $X$ is locally contractible. Now the lemma follows from \cite[Proposition IV.8.12]{dold1972lectures}. 
\end{proof}
\begin{lemma}
\label{lem:6.2}
The space $P_R(x_0)$ is a compact contractible metric space which is controlled $(2\dim(X)+1)$-dominated.
\end{lemma}
\begin{proof}
Because of \autoref{lem:6.1} we can find an open subset $U\subseteq \bbR^{2\dim(X)+1}$ and maps $i\colon X\to U$ and $r\colon U\to X$ with $r\circ i = \id_X$. Since $U$ is a smooth manifold, it can be triangulated and hence is a simplicial complex of dimension $(2\dim(X)+1)$. Since $P_R(x_0)$ is compact, $i(P_R(x_0))$ is compact and hence contained in a finite subcomplex $K\subseteq U$.
Let $i'\colon  P_R(x_0)\to K$ be the map defined by $i$. Let be $r'\colon K\to P_R(x_0)$ be the composite
\[K\xrightarrow{r|_K}X\xrightarrow{c_{x_0,\_}(R)} P_R(x_0).\]
Then $r'\circ i'=\id_{P_R(x_0)}$ and $K$ is a finite $(2\dim(X)+1)$-dimensional simplicial complex. This implies that $P_R(x_0)$ is controlled $(2\dim(X)+1)$-dominated.
\end{proof}
\section{The main theorem}
\label{sec:main}
\begin{thm}
\label{thm:main}
Let $G$ be a group as in \autoref{assumption}. Then the Farrell--Jones conjecture with finite wreath products holds for $G$.
\end{thm}
\begin{proof}
By \autoref{cor:cocompactac} and \autoref{lem:properac} the action of $G$ on $FS(X,\gamma)$ is isometric, proper and cocompact and hence $FS(X,\gamma)$ is a cocompact flow space for $G$. The flow space $FS(X,\gamma)$ is also finite-dimensional by \autoref{cor:findim} and it admits strong contracting transfers by \autoref{prop:transfer}. The main theorem now follows from the following theorem.
\end{proof}
\begin{thm}[{\cite[Corollary 9.9]{covers}}]
If $X$ is a cocompact, finite-dimensional flow space for the group
$G$, which admits strong contracting transfers, then $G$ is strongly transfer reducible with respect to the family $\VCyc$, in particular $G$ satisfies $\FJCw{}$.
\end{thm}
\bibliographystyle{amsalpha}
\bibliography{Busemann}

\providecommand{\bysame}{\leavevmode\hbox to3em{\hrulefill}\thinspace}
\providecommand{\MR}{\relax\ifhmode\unskip\space\fi MR }
\providecommand{\MRhref}[2]{%
  \href{http://www.ams.org/mathscinet-getitem?mr=#1}{#2}
}
\providecommand{\href}[2]{#2}
\begin{thebibliography}{{Weg}15}

\bibitem[AB95]{alonso1995semihyperbolic}
Juan~M Alonso and Martin~R Bridson, \emph{Semihyperbolic groups}, Proceedings
  of the London Mathematical Society \textbf{3} (1995), no.~1, 56--114.

\bibitem[BH99]{bh}
Martin~R Bridson and Andr{\'e} Haefliger, \emph{Metric spaces of non-positive
  curvature}, vol. 319, Springer, 1999.

\bibitem[BL12a]{bartels2012borel}
Arthur Bartels and Wolfgang L{\"u}ck, \emph{The {B}orel {C}onjecture for
  hyperbolic and {CAT}(0)-groups}, Annals of Mathematics \textbf{175} (2012),
  no.~2, 631--689.

\bibitem[BL12b]{flow}
\bysame, \emph{Geodesic flow for {C}{A}{T}(0)-groups}, Geom. Topol. \textbf{16}
  (2012), no.~3, 1345--1391. \MR{2967054}

\bibitem[BLR08a]{bartels2008equivariant}
{A}rthur {B}artels, {W}olfgang {L}{\"u}ck, and {H}olger {R}eich,
  \emph{{E}quivariant covers for hyperbolic groups}, {G}eom. {T}opol
  \textbf{12} (2008), no.~3, 1799--1882.

\bibitem[BLR08b]{bartels2008k}
Arthur Bartels, Wolfgang L{\"u}ck, and Holger Reich, \emph{The {K}-theoretic
  {F}arrell--{J}ones conjecture for hyperbolic groups}, Inventiones
  mathematicae \textbf{172} (2008), no.~1, 29--70.

\bibitem[BM97]{burger1997finitely}
Marc Burger and Shahar Mozes, \emph{Finitely presented simple groups and
  products of trees}, Comptes Rendus de l'Acad{\'e}mie des Sciences-Series
  I-Mathematics \textbf{324} (1997), no.~7, 747--752.

\bibitem[{Bow}95]{bowditch}
B.H. {Bowditch}, \emph{{Minkowskian subspaces of non-positively curved metric
  spaces.}}, {Bull. Lond. Math. Soc.} \textbf{27} (1995), no.~6, 575--584
  (English).

\bibitem[BR07]{coefficients}
Arthur Bartels and Holger Reich, \emph{Coefficients for the {F}arrell--{J}ones
  conjecture}, Advances in Mathematics \textbf{209} (2007), no.~1, 337--362.

\bibitem[DL15]{lang2}
Dominic Descombes and Urs Lang, \emph{Convex geodesic bicombings and
  hyperbolicity}, Geom. Dedicata \textbf{177} (2015), 367--384. \MR{3370039}

\bibitem[Dol95]{dold1972lectures}
Albrecht Dold, \emph{Lectures on algebraic topology}, Classics in Mathematics,
  Springer-Verlag, Berlin, 1995, Reprint of the 1972 edition. \MR{1335915
  (96c:55001)}

\bibitem[Eng78]{engelking1978dimension}
Ryszard Engelking, \emph{{D}imension theory}, North-Holland Publishing Company
  Amsterdam, 1978.

\bibitem[KR]{covers}
Daniel Kasprowski and Henrik R\"uping, \emph{Long and thin covers for cocompact
  flow spaces}, arXiv:1502.05001.

\bibitem[Lan13]{lang}
Urs Lang, \emph{Injective hulls of certain discrete metric spaces and groups},
  J. Topol. Anal. \textbf{5} (2013), no.~3, 297--331. \MR{3096307}

\bibitem[LR05]{baum}
Wolfgang L{\"u}ck and Holger Reich, \emph{The {B}aum-{C}onnes and the
  {F}arrell-{J}ones conjectures in {$K$}- and {$L$}-theory}, Handbook of
  {$K$}-theory. {V}ol. 1, 2, Springer, Berlin, 2005, pp.~703--842. \MR{2181833
  (2006k:19012)}

\bibitem[Mun00]{munkres}
James~R Munkres, \emph{Topology: a first course, 2nd edition}, 2000.

\bibitem[Nag77]{nagy1977tangent}
P.~T. Nagy, \emph{On the tangent sphere bundle of a {R}iemannian 2-manifold},
  T\^ohoku Math. J. \textbf{29} (1977), no.~2, 203--208. \MR{0474090 (57
  \#13745)}

\bibitem[Sco83]{scott83}
Peter Scott, \emph{The geometries of {$3$}-manifolds}, Bull. London Math. Soc.
  \textbf{15} (1983), no.~5, 401--487. \MR{705527 (84m:57009)}

\bibitem[Weg12]{wegnercat}
Christian Wegner, \emph{The {K}-theoretic {F}arrell-{J}ones conjecture for
  {CAT}(0)-groups}, Proceedings of the American Mathematical Society
  \textbf{140} (2012), no.~3, 779--793.

\bibitem[{Weg}15]{wegner2013farrell}
Christian {Wegner}, \emph{{The Farrell-Jones conjecture for virtually solvable
  groups.}}, {J. Topol.} \textbf{8} (2015), no.~4, 975--1016 (English).

\end{thebibliography}

\end{document}